\documentclass[11pt]{article}
\usepackage{latexsym}
\usepackage{amssymb,amsfonts,enumerate}
\usepackage{amsmath}
\usepackage{graphicx}
\usepackage[all]{xy}
\newtheorem{theorem}{Theorem}

\newtheorem{lemma}{Lemma}

\newtheorem{prop}{Proposition}
\newtheorem{cor}{Corollary}
\newtheorem{definition}{Definition}
\newtheorem{rmk}{Remark}

\newenvironment{proof}{\medskip \noindent
{\bf Proof.}}{\hfill \rule{.5em}{1em}
\\}
\newenvironment{proof1}{\medskip \noindent
{\bf Proof of Theorem 1}}{\hfill \rule{.5em}{1em}
\\}

\newenvironment{proof3}{\medskip \noindent
{\bf Proof of Theorem 3}}{\hfill \rule{.5em}{1em}
\\}

\newcommand\blfootnote[1]{%
  \begingroup
  \renewcommand\thefootnote{}\footnote{#1}%
  \addtocounter{footnote}{-1}%
  \endgroup
}

\def\bea{\begin{eqnarray*}}
\def\eea{\end{eqnarray*}}
\def\be{\begin{equation}}
\def\ee{\end{equation}}

\begin{document}

\title{On the classification of 4-dimensional $(m,\rho)$-quasi-Einstein manifolds with harmonic Weyl curvature}

\date{}

\author{ Jinwoo Shin }

\maketitle

\begin{abstract}
  In this paper we study 4-dimensional $(m,\rho)$-quasi-Einstein manifolds with harmonic Weyl curvature when $m\notin\{0,\pm1,-2,\pm\infty\}$ and $\rho\notin\{\frac{1}{4},\frac{1}{6}\}$. We prove that
a non-trivial $(m,\rho)$-quasi-Einstein metric $g$ (not necessarily complete) is locally isometric to one of the followings: (i) $\mathcal{B}^2_\frac{R}{2(m+2)}\times \mathbb{N}^2_\frac{R(m+1)}{2(m+2)}$ where $\mathcal{B}^2_\frac{R}{2(m+2)}$ is a northern hemisphere in the 2-dimensional sphere $\mathbb{S}^2_\frac{R}{2(m+2)}$, $\mathbb{N}_\delta$ is the
2-dimensional Riemannian manifold with constant curvature $\delta$ and $R$ is the constant scalar curvature of $g$, (ii)  $\mathcal{D}^2_\frac{R}{2(m+2)}\times\mathbb{N}^2_\frac{R(m+1)}{2(m+2)}$ where $\mathcal{D}^2_\frac{R}{2(m+2)}$ is
 one half (cut by a hyperbolic line) of the hyperbolic plane $\mathbb{H}^2_\frac{R}{2(m+2)}$, (iii) $\mathbb{H}^2_\frac{R}{2(m+2)}\times\mathbb{N}^2_\frac{R(m+1)}{2(m+2)}$ , (iv) a certain singular metric with $\rho=0$, (vi) a locally conformally flat metric. By applying this local classification, we obtain a classification of complete $(m,\rho)$-quasi-Einstein manifolds under the harmonic Weyl curvature condition. Our result can be viewed as a local classification of gradient Einstein-type manifolds. 

One corollary of our result is the classification of $(\lambda,4+m)$-Einstein manifolds which can be viewed as $(m,0)$-quasi-Einstein manifolds.
\end{abstract}

\blfootnote{
\textit{Date:} \today.

2010 \textit{Mathematics Subject Classification}. 53C21, 53C25.

\textit{Key words and phrases}. quasi-Einstein manifold, warped product, harmonic Weyl curvature, Codazzi tensor

This work was supported by the National Research Foundation of Korea(NRF) }

\section{Introduction}
\quad A Riemannian manifold $(M^n,g)$ is called a \textit{quasi-Einstein manifold}, if there exist a smooth function $f$ on $(M^n,g)$  and
two real constants $\mu,\lambda$ such that
\begin{equation}
Rc+\nabla df -\mu df\otimes df=\lambda g \label{quasidef}
\end{equation}
where $Rc$ denotes the Ricci tensor of $(M^n,g)$. One can easily see that if $f=const$ (we call this {\em trivial}) then this is nothing but an Einstein manifold and if $\mu=0$, this is a gradient Ricci soliton. Quasi-Einstein manifolds are especially interesting in that when $\mu=\frac{1}{m}$ for a positive integer $m$. In this case, if we take  $w=e^{-\frac{f}{m}}$
in (\ref{quasidef}), then  we obtain the following equations which is called the $(\lambda,n+m)$-Einstein manifold equation \cite{HPW}

\begin{equation}
  \nabla dw=\frac{w}{m}(Rc-\lambda g). \label{hpwdef}
\end{equation}
These manifolds are answers for one of the questions in the A.Besse's Book \cite[p. 265]{Be}. The question is whether one can construct Einstein metrics
which are warped products with a nonconstant warping function.
 When $m>1$, the $(\lambda,n+m)$-Einstein metric is exactly the base of an $n+m$ dimensional Einstein warped product,
i.e., $(E=M\times F^m,g_E=g+w^2 g_F)$ where $F$ is an $m$-dimensional Einstein manifold.

\bigskip

There are a number of remarkable studies of quasi-Einstein manifolds under various curvature conditions. It is known that according to G.Catino, C.Mantegazza, L.Mazzieri and M.Rimoldi \cite{CMMR},
a complete locally conformally flat quasi-Einstein manifold of dimension $n\geq3$ is locally a warped product with $(n-1)$-dimensional fibers of constant sectional curvature around any regular point of $f$ when
$\mu\neq\frac{1}{2-n}$. A complete simply connected $(\lambda,n+m)$-Einstein metric $(m>1)$ with harmonic Weyl tensor and the condition $W(\nabla w, \cdot,\cdot,\nabla w)=0$ is of the form $g=dt^2+\phi^2(t)g_L$, where
$g_L$ is an Einstein metric \cite{HPW}.  G.Catino \cite{Ca2} showed that a complete four-dimensional half conformally flat quasi-Einstein manifold with $\mu\neq -\frac{1}{2}$ is either Einstein or locally conformally flat. And Q.Chen and C.He \cite{CH} showed that a compact warped product Einstein manifold with vanishing Bach tensor of dimension $n\geq 4$ is either Einstein or a finite quotient of a warped product with an $(n-1)$-dimensional Einstein fiber.
\bigskip

Recently, G.Catino \cite{Ca} introduced a {\em generalized quasi-Einstein manifold}. He considered a manifold $(M^n,g)$,$n\geq 3$ satisfying
(\ref{quasidef}) for smooth functions $\mu$ and $\lambda$. G. Catino \cite{Ca} showed that a generalized quasi-Einstein manifold with harmonic Weyl tensor and $W(\nabla f,\cdot
,\cdot,\cdot)=0$ is locally a warped product with $(n-1)$-dimensional Einstein fibers around any regular point of $f$.

In this paper, we study {\em $(m,\rho)$-quasi-Einstein manifolds} which is a special case of generalized quasi-Einstein manifolds.
\begin{definition}\cite{GH}
  If there exist a smooth function $f$ on $(M^n,g)$ and three real constants $m,\rho,\lambda$ ($m\notin\{0,\pm\infty\})$ such that
\begin{equation}
  Rc+\nabla df-\frac{1}{m}df\otimes df=(\rho R+\lambda)g \label{rhodef}
\end{equation}
then we call $(M^n,g)$ a $(m,\rho)$-quasi-Einstein manifold, where $Rc$ and $R$ denote the Ricci curvature and the scalar curvature of $(M^n,g)$.
\end{definition}

G.Huang and Y.Wei \cite{GH} obtained some results on compact $(m,\rho)$-quasi-Einstein manifolds under the Bach flat condition.
Note that we allow $m$ to be negative but restrict to finite.
It is known that if $m=2-n$, then $(M^n,g,f)$ is conformal to an Einstein metric \cite{JW}. Thus in this paper we only consider the case $m\neq 2-n$.
Then we can regard our space as a {\em nondegenerate gradient Einstein-type manifold} \cite{CP} which
is a Riemannian manifold $(M,g,f)$ satisfying $\alpha Rc+\beta \nabla df+\mu df\otimes df=(\rho R+\lambda)$ for some $\alpha,\beta,\mu,\rho\in \mathbb{R}$ such that $\beta\neq0,\beta^2\neq(n-2) \alpha \mu$. 

\bigskip

The aim of this paper is to classify 4-dimensional $(m,\rho)$-quasi-Einstein manifolds under the weaker condition than the above work in \cite{Ca}. We only assume the harmonic Weyl curvature condition,
without the zero radial condition $W(\nabla f, \cdot,\cdot, \cdot)=0$. In this paper we do not assume that $M$ is complete. We will mainly consider this problem in local sense. The proof is motivated by J.Kim's \cite{JS} paper which is the corresponding result on gradient Ricci solitons.
\bigskip

  The following is the main theorem of this paper.
\begin{theorem}\label{mainthm}
  Let $(M^4,g,f)$ be a $(m,\rho)$-quasi-Einstein manifold (not necessarily complete) with harmonic Weyl curvature, $m\notin\{\pm1,-2\}$ and $\rho\notin\{\frac{1}{4},\frac{1}{6}\}$. Then there exists an open dense subset $U$ of $M$ such that
for each point $p$ in $U$ there exists a neighborhood $V_p$  where $(V_p,g)$ is isometric to one of the following;
 Here $R$ is the constant scalar curvature of $g$ and $C$ is an arbitrary constant.

{\rm (i)} An Einstein metric with $f$ a constant function.

 {\rm (ii)} A domain in $\mathcal{B}_{\frac{R}{2(m+2)}}^2 \times\mathbb{N}_{\frac{R(m+1)}{2(m+2)}}^2$ with $f=-m\ln(\cos\sqrt{\frac{R}{2(m+2)}}s)+C$ and $\frac{R}{2(m+2)}>0$
where  $\mathcal{B}_{\frac{R}{2(m+2)}}^2$ is the  northern hemisphere in the 2-dimensional sphere $\mathbb{S}_{\frac{R}{2(m+2)}}^2$, $\mathbb{N}^2_\delta$ is a 2-dimensional
Riemannian manifold with constant curvature $\delta$ and
 $s\in(0,\frac{\pi}{2})$ is the distance function on $\mathcal{B}^2_\frac{R}{2(m+2)}$ from the the north pole.

 {\rm (iii)}A domain in $\mathcal{D}_{\frac{R}{2(m+2)}}^2 \times\mathbb{N}_{\frac{R(m+1)}{2(m+2)}}^2$ with  $f=-m\ln(-\sinh\sqrt{-\frac{R}{2(m+2)}}s)+C$ and  $\frac{R}{2(m+2)}<0$ where $\mathcal{D}_{\frac{R}{2(m+2)}}^2$ is the set $\{(s,t)|s<0\}$ in 2-dimensional the hyperbolic plane $\mathbb{H}_{\frac{R}{2(m+2)}}^2$  with
the metric $g=ds^2+\cosh^2\Big(\sqrt{-\frac{R}{2(m+2)}}s\Big)dt^2$ and $s\in(-\infty,0)$ can be viewed as the signed distance function on $\mathcal{D}^2_\frac{R}{2(m+2)}$ from the line $\{(s,t)|s = 0\}$

 {\rm (iv)} A domain in $\mathbb{H}_{\frac{R}{2(m+2)}}^2 \times\mathbb{N}_{\frac{R(m+1)}{2(m+2)}}^2$ with $f=-m\ln(\cosh\sqrt{-\frac{R}{2(m+2)}}s)+C$ and  $\frac{R}{2(m+2)}<0$,
where $s\in(-\infty,0)$ is the signed distance function on $\mathbb{H}^2_\frac{R}{2(m+2)}$ from a point.

{\rm (v)} A domain in $\mathbb{R}^4=\{s,t,x_3,x_4\}\setminus\{s\leq0\}$ with the metric $g=ds^2+s^{\frac{2(m-1)}{3(m+1)}}dt^2+s^{\frac{4}{3}}\hat{g}$, where
$\hat{g}$ is the pull-back of the Euclidean metric on the $(x_3,x_4)-$plane. Here $\rho=\lambda=0$, $m(m+1)>0$ and $f=\frac{2m}{3(m+1)}\ln s+C$.

{\rm (vi)}  A warped product with the metric of the form $g=ds^2+h(s)^2\bar{g}$ for a positive function $h$, where the Riemannian metric $\bar{g}$ has constant curvature. In
particular, $g$ is locally conformally flat.

\end{theorem}

 Our result can be viewed as a classification of 4-d nondegenerate gradient Einstein-type manifolds with $\mu\neq0$ under the harmonic Weyl curvature condition.

We first show that Ricci eigenvalues depend only on one variable $s=\int\frac{\nabla f}{|\nabla f|}$ due to harmonic Weyl curvature condition.
Then most geometric quantities involved are also functions of $s$ only. And then we define a function $\zeta_i(s):=<\nabla_{E_i}E_1,E_i>$ where $\{E_1=\frac{\nabla f}{|\nabla f|},\cdots,E_4\}$ is a orthonormal Ricci-eigen frame.
 We will consider this problem as three divided cases according to the distinctiveness of Ricci eigenvalues.
For each case, we can express Riemannian curvatures and the potential functin $f$ into $\zeta_i$ thanks to A.Derdzinski's lemma\cite{Derd}. Putting these
expression into the $(m,\rho)$-quasi-Einstein equation and the harmonic Weyl equation, we will get several ordinary differential equations of $\zeta_i$.
Analyzing these ODEs, we can get a number of  possible relations between $\zeta_i$ and then we are able to solve these ODEs for each cases.

\bigskip

From our main result, we can get a classification of complete $(m,\rho)$-quasi-Einstein manifolds and also a local classification under harmonic curvature condition.
 In the last section, we will give a classification complete $(\lambda,4+m)$-Einstein manifolds which can be viewed as $(m,0)$-quasi-Einstein manifolds.  In this context, a complete manifold means a manifold with boundary which is Cauchy complete.
\begin{cor}\label{completecase}
   A (Cauchy) complete $(\lambda,4+m)$-Einstein manifold $(M^4,g,w=e^{-\frac{f}{m}})$ with harmonic Weyl curvature and $m>1$ is is isometric to one of the following;
Here $C$ is an arbitrary positive constant.

  {\rm (i)} $g$ is an Einstein metric with $w$ a constant function.

  {\rm (ii)} A finite quotient of $\mathcal{B}_{\frac{\lambda}{m+1}}^2\times\mathbb{S}_{\lambda}^2$ with $\lambda>0$ and
 $w=e^{-\frac{f}{m}}=C\cos\sqrt{\frac{\lambda}{m+1}}s$.

{\rm (iii)} A finite quotient of $\mathcal{D}_{\frac{\lambda}{m+1}}^2\times\mathbb{H}_{\lambda}^2$ with $\lambda<0$ and
$w=-C\sinh\sqrt{\frac{-\lambda}{m+1}}s$.

{\rm (iv)} A finite quotient of $\mathbb{H}_{\frac{\lambda}{m+1}}^2\times\mathbb{H}_{\lambda}^2$ with $\lambda<0$ and
$w=C\cosh\sqrt{\frac{-\lambda}{m+1}}s$.

{\rm (v)} $g$ is locally conformally flat.
\end{cor}

This paper is organized as follows. In section 2, we study some known properties of $(m,\rho)$-quasi-Einstein manifolds with harmonic Weyl tensor.
In section 3, we discuss the case that all eigenvalues of Ricci tensor are mutually distinct. In section 4, we classify the case that two of them are equal i.e., $\lambda_3=\lambda_4$.
In section 5, we study the remaining case $\lambda_2=\lambda_3=\lambda_4$ and prove our theorems by combining previous sections. In section 6, we give a classification of $(\lambda,4+m)$-Einstein manifolds.

\section{$(m,\rho)$-quasi-Einstein manifolds with harmonic Weyl curvature}

\quad In this section we fix our notations and discuss some basic facts and known results about $(m,\rho)$-quasi-Einstein manifolds  with harmonic Weyl curvature.

The Riemanian curvature is defined  by
\begin{equation*}
  R(X,Y,Z,W)=<\nabla_X\nabla_Y Z-\nabla_Y\nabla_X Z-\nabla_{[X,Y]}Z,W>
\end{equation*}

Let $(M^n,g,f)$ be a $(m,\rho)$-quasi-Einstein manifold with the harmonic Weyl curvature condition. Then we obtain the following equation \cite{GH} from the well-known equality
$(\nabla_X H_f)(Y,Z)-(\nabla_Y H_f)(X,Z)=R(X,Y,\nabla f,Z)$.
\begin{align}
 R(X,Y,\nabla f,Z)=\Big\{\rho-\frac{1}{2(n-1)}\Big\}\{X(R)g(Y,Z)-Y(R)g(X,Z)\}  \nonumber\\
-\frac{1}{m}df(X)\nabla df(Y,Z) +\frac{1}{m}df(Y)\nabla df(X,Z)  \label{basiceq}
\end{align}

As mentioned in the introduction, if $m=2-n$, $(M^n,g)$ is conformal to an  Einstein metric. Thus from now on, we only consider the case $m\neq 2-n$.
 \begin{lemma} \label{threesolb}
For a  $(m,\rho)$-quasi-Einstein manifold $(M^n, g, f)$, $n\geq 3$ with the harmonic Weyl curvature condition, $\rho\neq \frac{1}{2(n-1)}$ and $m\neq2-n$,  let $c$ be a regular value of $f$ and $\Sigma_c= \{ x | f(x) =c  \}$  be the level surface of $f$. Then the followings hold;

{\rm (i)} Where $\nabla f \neq 0$,  $E_1 := \frac{\nabla f }{|\nabla f  | }$ is an eigenvector field of $Rc$.

{\rm (ii)} The scalar curvature $R$ and $ |\nabla f|^2$  are constant on a connected component of $\Sigma_c$.

{\rm (iii)} There is a function $s$ locally defined with   $s(x) = \int  \frac{   d f}{|\nabla f|} $, so that $ds =\frac{   d f}{|\nabla f|}$ and $E_1 = \nabla s$.

{\rm (iv)}  $R({E_1, E_1})$ is constant on a connected component of $\Sigma_c$.

{\rm (v)} $\nabla _{E_1}E_1=0$

\end{lemma}
\begin{proof}
  We only prove that $Rc(X,\nabla f)=0$ when $X\perp \nabla f$ in this proof. For other details, one may follow the argument in the proof of Lemma 3.3 in \cite{CC1} or see \cite{HPW}. In any neighborhood, where $|\nabla f|\neq0$, of a level set $\Sigma_c=\{p\in M^n|f(p)=c\}$, we can express the metric $g$ as  $g=\frac{1}{|\nabla f|^2}df\otimes df+g_{ij}(f,\theta)d\theta^i\otimes d\theta^j$ where $\theta=(\theta^2,\ldots,\theta^{n})$ denotes coordinates for $\Sigma_c$. Take $\partial_1=\frac{\nabla f}{|\nabla f|^2}$ and
  $\partial_i=\frac{\partial}{\partial\theta_i}$ for $i\geq 2$ and suppose $\rho\neq \frac{1}{2n-2}$. Then by direct computations, we have
\begin{align*}
  &\nabla_{\partial_i}|\nabla f|^2=-2|\nabla f|^2R_{1i}\\
  &\nabla_{\partial_1}|\nabla f|^2=-2|\nabla f|^2R_{11}+\frac{2}{m}|\nabla f|^2+2(\rho R+\lambda)\\
  &\{(n-1)\rho-\frac{1}{2}\}\nabla_{\partial_i}R=(\frac{1}{m}-1)|\nabla f|^2R_{1i}\\
  &\{(n-1)\rho-\frac{1}{2}\}\nabla _{\partial_1} R=(\frac{1}{m}-1)|\nabla f|^2 R_{11}+\frac{1}{m}(\rho R+\lambda)(n-1)-\frac{R}{m}\\
  &\nabla_1R_{i1}-\nabla_i R_{11}=\frac{2(m+n-2)\rho-1}{m\{2(n-1)\rho-1\}}R_{i1}
\end{align*}
Since $M^n$ satisfies the harmonic Weyl condition, we also have $\nabla_1 R_{i1}-\nabla_i R_{11}=\frac{\nabla_1 R}{2n-2}g_{i1}-\frac{\nabla_i R}{2n-2}g_{11}=\frac{m-1}{m(n-1)\{2(n-1)\rho-1\}}R_{i1}$. Hence we get $\frac{2-m-n}{m(n-1)}R_{i1}=0$. Since $m\neq2-n$, if $n\geq 3$ then we have $R_{i1}=0$.
\end{proof}

 A crucial factor that makes J.Kim's method works is the fact that a manifold which satisfies harmonic Weyl condition has a Codazzi tensor.
A Codazzi tensor $\mathcal{C}$ on a Riemannian manifold $M$ is a symmetric tensor of covariant order 2 such that $d^{\nabla} \mathcal{C}=0$, which can be written in local
coordinates as $\nabla_k \mathcal{C}_{ij}=\nabla_{i}\mathcal{C}_{kj}$. The harmonic Weyl curvature condition $\delta W=0$ can be written as $d^\nabla(Rc-\frac{R}{2n-2}g)=0$.
Therefore, $\mathcal{T}:=Rc-\frac{R}{2n-2}g$ is a Codazzi tensor when $M$ satisfies the harmonic Weyl condition.  A.Derdzinski \cite{Derd} described properties about a Codazzi tensor as following; For a Codazzi tensor $\mathcal{C}$ and a point $x$ in $M$,
let $E_{\mathcal{C}}(x)$ be the number of distinct eigenvalues of $\mathcal{C}_x$, and set $M_\mathcal{C}:=\{x\in M|E_\mathcal{C}\textrm{ is constant in a neighborhood of }x\}$,
then $M_\mathcal{C}$ is an open dense subset of $M$ and that in each connected component of $M_\mathcal{C}$, the eigenvalues are well-defined and differentiable
functions \cite{Derd}.

\begin{lemma}
  Let $(M^n,g,f)$, $n\geq 3$, be a $(m,\rho)$-quasi-Einstein manifold with $\rho\notin\{\frac{1}{n},\frac{1}{2(n-1)}\}$. Then, in harmonic coordinates, the metric $g$ and the
function $f$ are real analytic.
\end{lemma}
\begin{proof}
  One can follow the argument in the proof of [\cite{CM}, Theorem 2.4]
\end{proof}

  Thus
if $f$ is not a constant and $\rho\notin\{\frac{1}{n},\frac{1}{2(n-1)}\}$, then $\{\nabla f\neq0\}$ is open and dense in $M$. So  $M_{\mathcal{T}}\cap\{\nabla f\neq0\}$ is an open dense subset of $M$.

 We consider orthonormal Ricci-eigen vector fields $E_i$ in a
neighborhood of each point in $M_{\mathcal{T}}\cap\{\nabla f\neq0\}$. Let $E_1=\frac{\nabla f}{|\nabla f|}$, then for $i>1$, $E_i$ is tangent to smooth level hypersurfaces of $f$.
 We call these local orthonormal Ricci-eigen vector fields $E_i$ an \emph{adapted frame field} of $(M^n,g,f)$ and denote $R_{ij}:=R(E_i,E_j)=\lambda_i\delta_{ij}$.

 From Lemma \ref{threesolb}, in a neighborghood
of a point $p\in M_{\mathcal{T}}\cap\{\nabla f\neq 0\}$, $f$ and $R$ may be considered as functions of the variable $s$ only. Actually this is not just for $R$,
all eigenvalues of Ricci tensor depend only on $s$. We will prove this in the next Lemma. We write the derivative in $s$ by a prime. Also note that $df(E_1)=g(\nabla f,\frac{\nabla f}{|\nabla f|})=|\nabla f|$.
So, $|\nabla f|=f'$. Then from the equation (\ref{basiceq}) we have
\begin{equation}
  R_{1ii1}=\frac{R'}{f'}(\frac{1}{2n-2}-\rho)-\frac{1}{m}(R_{ii}-\rho R-\lambda) \label{R1ii1}
\end{equation}
And from the $(m,\rho)$-quasi-Einstein equation (\ref{rhodef}), for $i>1$, we can get
\begin{equation}
 \nabla_{E_i}E_1=:\zeta_i E_i=\frac{1}{f'}(\rho R+\lambda - R_{ii})E_i
\end{equation}

 \begin{lemma}\label{depends}
For a $(m,\rho)$-quasi-Einstein manifold $(M^4,g,f)$ with harmonic Weyl curvature, $\rho\neq \frac{1}{6}$ and $m\neq -2$, and for a local adapted frame field $\{E_i\}$ in an open subset of $M_{\mathcal{T}}\cap\{\nabla f\neq 0\}$,
the Ricci-eigen functions $\lambda_i$, $i=1,\cdots,n$, are constant on a connected component of a regular level hypersurface $\Sigma_c$ of $f$, and
 depend on the local variable $s$ only.

 \end{lemma}

\begin{proof}
We will show that $\sum_{i=1}^n \lambda_{i}^k$ is a function of $s$ only for all $k$. For $k=1$, it is already proved in Lemma \ref{threesolb}. Suppose
it is true for all $j<k+1$. Then we have
\begin{align*}
  E_1\big\{\sum_{i=1}^n(R_{ii}^k)\big\}=&\sum_{i=1}^nk R_{ii}^{k-1}E_1(R_{ii})\\
  =&\sum_{i=1}^nk R_{ii}^{k-1}\nabla_{E_1}R_{ii}\\
  =&\sum_{i=1}^nk R_{ii}^{k-1}(\nabla_{E_i}R_{i1}+\frac{1}{6}R')\\
  =&k(R_{11}'+\frac{1}{6}R')R_{11}^{k-1}+\sum_{i=2}^n kR_{ii}^{k-1}\{\zeta_i(R_{11}-R_{ii})+\frac{1}{6}R'\}\\
  =&kR_{11}'R_{11}^{k-1}+\frac{kR'}{6}\sum_{i=1}^n R_{ii}^{k-1}\\
    &\qquad\qquad\qquad+\frac{k}{f'}\sum_{i=2}^n R_{ii}^{k-1}(R_{11}-R_{ii})(\rho R+\lambda- R_{ii})
\end{align*}

Due to assumption, every term except $R_{ii}^{k+1}$ in the above equation depends only on $s$.
So $\sum_{i=1}^n R_{ii}^{k+1}$ is also a function of $s$ only. Therefore by the mathematical induction, $\sum_{i=1}^nR_{ii}^k=\sum_{i=1}^n\lambda_i^k$ depends only on $s$ for all $k=1,2,\ldots$.
This implies that each $\lambda_i$, $i=1,\cdots,n$, is a constant depending only on $s$.
\end{proof}

 The following is Lemma 9 in \cite{JS} which is originated from A.Derdzinski's Lemma in \cite{Derd}.
This Lemma enable us to compute Riemannian curvature explicitly.

\begin{lemma} \label{abc60} For a Riemannian metric with harmonic Weyl curvature, consider orthonormal vector fields $E_i$, $i=1, \cdots n$ in an open set such that
$Ric(E_i) = \lambda_i E_i$.  Then $\mathcal{T}(E_i) = (\lambda_i - \frac{R}{2n-2} ) E_i$ and the following holds;

{\rm (i)}
 $(\lambda_j - \lambda_k ) ( \nabla_{E_i} E_j, E_k)   + \nabla_{E_i}(  E_k, \mathcal{T}E_j)=(\lambda_i - \lambda_k ) ( \nabla_{E_j} E_i, E_k) +\nabla_{E_j}(  E_k, \mathcal{T} E_i), \ \ $ for any $i,j,k =1, \cdots n$.

{\rm (ii)}  for  distinct $i,j,k \geq 1$ (with not necessarily distinct $\lambda_i, \lambda_j, \lambda_k$),  it holds that
$ \ \ \ \ (\lambda_j - \lambda_k ) ( \nabla_{E_i} E_j, E_k)=(\lambda_i - \lambda_k ) ( \nabla_{E_j} E_i, E_k) .$

\end{lemma}

\begin{lemma}\label{77}
Let $(M^4,g,f)$ be a $(m,\rho)$-quasi-Einstein manifold with harmonic Weyl curvature, $\rho\neq\frac{1}{6}$ and $m\neq-2$. Consider an adapted frame fields $E_j$, $j=1,2,3,4$,
in an open subset $O$ of $M_{\mathcal{T}}\cap\{\nabla f\neq 0\}$. Then the following hold in $O$;

{\rm (i)}If the eigenfunctions $\lambda_2,\lambda_3,\lambda_4$ are distinct from each other, then

$\nabla_{E_i}E_i=-\zeta_i E_1$ and $\nabla_{E_1}E_i=0$.

{\rm (ii)}If $ \lambda_2  \neq \lambda_3=  \lambda_4$, then

$\nabla_{E_2}  E_2 = -\zeta_2(s) E_1  ,\nabla_{E_3}  E_3 = -\zeta_3  E_1 -  \beta_3  E_4  , \nabla_{E_4}  E_4 = -\zeta_4  E_1 + \beta_4 E_3$

$\nabla_{E_1}  E_2=\nabla_{E_3}  E_2=\nabla_{E_4}  E_2=0$,$\nabla_{E_3}  E_4 = \beta_3 E_3, \nabla_{E_4}  E_3 = - \beta_4 E_4 $

$ \nabla_{E_1}  E_3= q E_4, \nabla_{E_1}  E_4= -q E_3,\nabla_{E_2}  E_3=   r E_4 , \nabla_{E_2}  E_4=  - r E_3$

\end{lemma}
\begin{proof}
  It follows from Lemma \ref{abc60} and direct computation.
\end{proof}

\section{$(m,\rho)$-quasi-Einstein manifolds with distinct  $\lambda_2, \lambda_3, \lambda_4$}

\quad

 As mentioned in an introduction, we consider three cases depending on the distinctness of Ricci eigenvalues.  In this section we shall study the first case that all $\lambda_i$ are pairwise distinct. We will prove that if $\lambda_i$, $i=2,3,4$ are mutually different then $g$ should be an Einstein metric, by showing $f'=0$.

Note that from Lemma \ref{abc60}, $\Gamma_{jk}^i:=<\nabla_{E_j}E_k,E_i>$, $i,j,k=2,3,4$ is a function of $s$ only. So from the above Lemma \ref{77}, we may write
\begin{equation} \label{coeff01}
[E_2, E_3] = \alpha E_4, \ \ \  [E_3, E_4] = \beta E_2,  \ \ \ [E_4, E_2] = \gamma E_3.
\end{equation}

Then from the Jacobi identity of Lie bracket $[[E_1,E_2],E_3]+[[E_2,E_3],E_1]+[[E_3,E_1],E_2]=0$, we have the following relation.
\begin{align}
&\alpha^{'} = \alpha ( \zeta_4 -\zeta_2-\zeta_3   ),  \ \    \beta^{'} = \beta ( \zeta_2-\zeta_3-\zeta_4    ), \ \   \gamma^{'} = \gamma ( \zeta_3-\zeta_2-\zeta_4  ) \label{alphaprime} \\
&\beta =   \frac{(\zeta_3 - \zeta_4   )^2}{(\zeta_2 - \zeta_3   )^2} \alpha,   \ \      \gamma =   \frac{(\zeta_2 - \zeta_4   )^2}{(\zeta_2 - \zeta_3   )^2} \alpha. \hspace{3cm} \nonumber
\end{align}
Now we compute Riemannian curvatures.
\begin{align*}
\textrm{For }i,j,k>1\\
R_{1ii1} =&-\zeta_i^{'}  -  \zeta_i^2\\
R_{jiij} =& -\zeta_j \zeta_i  + \Gamma_{ij}^k\Gamma_{ji}^k-\Gamma_{jk}^i\Gamma_{kj}^i-\Gamma_{ik}^j\Gamma_{ki}^j\\
R_{1ij1}=&  R_{ijj1}=R_{1234}=0\\
R_{11} = &-\zeta_2'-\zeta_2^2-\zeta_3'-\zeta_3^2-\zeta_4'-\zeta_4^2\\
R_{22} = & -\zeta_2^{'}
 -  \zeta_2^2 -\zeta_2 \zeta_3
 -\zeta_2 \zeta_4  -2 \Gamma_{34}^2 \Gamma_{43}^2 \\
\end{align*}
\begin{align*}
R_{33} =& -\zeta_3^{'}
 -  \zeta_3^2 -\zeta_3 \zeta_2
 -\zeta_3 \zeta_4  +2 \frac{( \zeta_2 - \zeta_4 )}{ \zeta_3 - \zeta_4   }\Gamma_{34}^2 \Gamma_{43}^2 \\
R_{44} =& -\zeta_4^{'}
 -  \zeta_4^2 -\zeta_4 \zeta_2
 -\zeta_4 \zeta_3  +2\frac{( \zeta_2 - \zeta_3 )}{ \zeta_4 - \zeta_3   } \Gamma_{34}^2 \Gamma_{43}^2 \\
\end{align*}

We will express $f'$ into $\zeta_i$ to show that $f'=0$. From  $\zeta_i=\frac{1}{f'}(\rho R+\lambda-R_{ii})$, for $i>1$, we get
\begin{align*}
  f'=&\frac{R_{22}-R_{33}}{\zeta_2-\zeta_3}=-\frac{\zeta_2'-\zeta_3'}{\zeta_2-\zeta_3}-(\zeta_2+\zeta_3+\zeta_4)-2\Gamma_{34}^2\Gamma_{43}^2\frac{\zeta_2+\zeta_3-2\zeta_4}
{(\zeta_3-\zeta_4)(\zeta_2-\zeta_3)}\\
=&\frac{R_{22}-R_{44}}{\zeta_2-\zeta_4}=-\frac{\zeta_2'-\zeta_4'}{\zeta_2-\zeta_4}-(\zeta_2+\zeta_3+\zeta_4)-2\Gamma_{34}^2\Gamma_{43}^2\frac{\zeta_2+\zeta_4-2\zeta_3}
{(\zeta_4-\zeta_3)(\zeta_2-\zeta_4)}
\end{align*}
Hence, we have
\begin{equation}
  \Gamma_{34}^2\Gamma_{43}^2=\frac{(\zeta_2-\zeta_3)(\zeta_3-\zeta_4)(\zeta_4-\zeta_2)}{4
(\zeta_2^2+\zeta_3^2+\zeta_4^2-\zeta_2\zeta_3-\zeta_2\zeta_4-\zeta_3\zeta_4)}\Big(\frac{\zeta_2'-\zeta_3'}{\zeta_2-\zeta_3}-\frac{\zeta_2'-\zeta_4'}
{\zeta_2-\zeta_4}\Big)\label{expgamma}
\end{equation}
We also have $ m(R_{1ii1}-R_{1jj1})=R_{jj}-R_{ii}$     from (\ref{R1ii1}). Thus,
\begin{eqnarray*}
  \frac{\zeta_2'-\zeta_3'}{\zeta_2-\zeta_3}=-(\zeta_2+\zeta_3)-\frac{\zeta_4}{m+1}-2\frac{\Gamma_{34}^2\Gamma_{43}^2}{(m+1)}\frac{(\zeta_2+\zeta_3-2\zeta_4)}
{(\zeta_3-\zeta_4)(\zeta_2-\zeta_3)}\\
  \frac{\zeta_2'-\zeta_4'}{\zeta_2-\zeta_4}=-(\zeta_2+\zeta_4)-\frac{\zeta_3}{m+1}-2\frac{\Gamma_{34}^2\Gamma_{43}^2}{(m+1)}\frac{(\zeta_2+\zeta_4-2\zeta_3)}
{(\zeta_4-\zeta_3)(\zeta_2-\zeta_4)}\\
\frac{\zeta_3'-\zeta_4'}{\zeta_3-\zeta_4}=-(\zeta_3+\zeta_4)-\frac{\zeta_2}{m+1}-2\frac{\Gamma_{34}^2\Gamma_{43}^2}{m+1}\frac{(\zeta_3+\zeta_4-2\zeta_2)}{(\zeta_3-\zeta_4)^2}
\end{eqnarray*}
Therefore, from (\ref{expgamma}), we get
\begin{eqnarray}
  \Gamma_{34}^2\Gamma_{43}^2=\frac{(\zeta_2-\zeta_3)(\zeta_2-\zeta_4)(\zeta_3-\zeta_4)^2}
{4(\zeta_2^2+\zeta_3^2+\zeta_4^2-\zeta_2\zeta_3-\zeta_3\zeta_4-\zeta_2\zeta_4)} \hspace{2.3cm}\label{expgamma2}\\
\frac{\zeta_i'-\zeta_j'}{\zeta_i-\zeta_j}=-(\zeta_i+\zeta_j)-\frac{\zeta_k}{m+1}-\frac{1}{2(m+1)}\frac{(\zeta_i-\zeta_k)(\zeta_j-\zeta_k)(\zeta_i+\zeta_j-2\zeta_k)}
{(\zeta_2^2+\zeta_3^2+\zeta_4^2-\zeta_2\zeta_3-\zeta_2\zeta_4-\zeta_3\zeta_4)} \nonumber \\
f'=-\frac{\zeta_2^2\zeta_3+\zeta_2\zeta_3^2+\zeta_2^2\zeta_4+\zeta_2\zeta_4^2+\zeta_3^2\zeta_4+\zeta_3\zeta_4^2-6\zeta_2\zeta_3\zeta_4}{2(m+1)
(\zeta_2^2+\zeta_3^2+\zeta_4^2-\zeta_2\zeta_3-\zeta_3\zeta_4-\zeta_4\zeta_2)} \hspace{1.8cm} \label{expofwprime}
 \end{eqnarray}
\begin{prop}\label{1111main}
  Let $(M^4,g,f)$ be a $(m,\rho)$-quasi-Einstein manifold with harmonic Weyl curvature, $\rho\neq\frac{1}{6}$ and $m\neq-2$. For any adapted frame fields $E_j$, $j=1,2,3,4$, in an open subset $O$ of
  $M_{\mathcal{T}}\cap\{\nabla f\neq0\}$ , if the three eigenfunctions $\lambda_2,\lambda_3,\lambda_4$ are pairwise distinct, then $f$ must be a constant function, so
$g$ is an Einstein metric.
\end{prop}
\begin{proof}
Suppose that $\lambda_2,\lambda_3,\lambda_4$ are pairwise distinct.
 In this proof we set $a= \zeta_2$, $b= \zeta_3$, $c= \zeta_4$ and
$P:= a^2  +b^2+ c^2 - ab  -  bc - ac= \frac{1}{2} \{ (a-b)^2 + (a-c)^2 + (b-c)^2  \} $ for a notational convenience.
From (\ref{expgamma2}) and Lemma \ref{abc60},
$ (\alpha - \gamma + \beta)^2= 4(\Gamma_{34}^2)^2  =\frac{(a - b)^2 (b - c )^2}{ ( a^2  +b^2+ c^2 - ab  -  bc - ac ) }$.
From the equation (\ref{alphaprime}) we have,
\begin{equation*}
(\alpha - \gamma + \beta)^2 =\alpha^2\{1 -  \frac{(a - c   )^2}{(a - b   )^2}  +  \frac{(b - c   )^2}{(a - b   )^2} \}^2  = \frac{4 \alpha^2 (b  - c )^2}{(a - b   )^2}.
 \end{equation*}

So,   $ \alpha^2    = \frac{(a - b)^4 }{ 4P } .$ Since $a,b,c$ are all functions of $s$ only, so is $\alpha$. Differentiating this in $s$ and using the above, we get

 \begin{eqnarray*}
 2 \alpha \alpha^{'}    =  &  \frac{4(a - b)^3 (a^{'} - b^{'}) }{ 4P}  -   \frac{(a - b)^4 ( 2a a^{'}  +2b b^{'} + 2c c^{'}  - a b^{'} - b a^{'} -  a c^{'} -c a^{'} - c b^{'}- b c^{'}  )  }{ 4P^2 } \hspace{1.5cm} \\
 = &  \frac{4(a - b)^3 (a^{'} - b^{'}) }{ 4P}  -   \frac{(a - b)^4 \{ (a -b)(a^{'} - b^{'}) + (a-c)( a^{'} - c^{'})+ (b-c)( b^{'} - c^{'})   \}  }{ 4P^2 } \hspace{1.5cm} \\
  = &  \frac{4(a - b)^4 \frac{(a^{'} - b^{'})}{a-b} }{ 4P}  -   \frac{(a - b)^4 \{ (a -b)^2\frac{(a^{'} - b^{'})}{a-b} + (a-c)^2\frac{(a^{'} - c^{'})}{a-c}+ (b-c)^2\frac{(b^{'} - c^{'})}{b-c}   \}  }{ 4P^2 }. \hspace{1.5cm} \\
 \end{eqnarray*}
Thus,
 \begin{align}
 \frac{8P^2 \alpha \alpha^{'}}{(a-b)^4}
  = &  4 P \frac{(a^{'} - b^{'})}{a-b}   -   \{ (a -b)^2\frac{(a^{'} - b^{'})}{a-b} + (a-c)^2\frac{(a^{'} - c^{'})}{a-c}+ (b-c)^2\frac{(b^{'} - c^{'})}{b-c}   \}\nonumber  . \hspace{1.5cm} \\
=&2P(c-a-b)-\frac{m+2}{m+1}(a^2b+ab^2+a^2c+ac^2+b^2c+bc^2-6abc) \label{alpha221}
 \end{align}

Meanwhile, from $\frac{\alpha^{'}}{\alpha } =  ( c-a-b    )$,

\begin{equation*}
 2 \alpha \alpha^{'}=  2 \alpha^2 \frac{\alpha^{'}}{\alpha } = 2 \alpha^2 ( c-a-b    )=2 \frac{(a - b)^4 }{ 4P }( c-a-b   ).
 \end{equation*}

\begin{equation} \label{ptwo}
\frac{8P^2 \alpha \alpha^{'}}{(a-b)^4}=2P( c-a-b   ).
 \end{equation}

Equating (\ref{alpha221}) and (\ref{ptwo}), we get;
\begin{equation}
  -\frac{m+2}{m+1}(a^2b+ab^2+a^2c+ac^2+b^2c+bc^2-6abc)=0
\end{equation}

From this we get  $6abc = ab^2 + ba^2 + ac^2 + ca^2 + bc^2 + cb^2   $.

Finally from (\ref{expofwprime}), we get $f' =0$. Therefore, $g$ is an Einstein metric.
\end{proof}
\section{$(m,\rho)$-quasi-Einstein manifold with $\lambda_2 \neq \lambda_3 =\lambda_4$. }

\quad In this section we shall study the second case when two of Ricci eigenfunctions are equal. We may well assume that $ \lambda_2  \neq \lambda_3=  \lambda_4 $.

Consider the 2-dimensional distributions $D^1$ and $D^2$ which is spanned by $E_1,E_2$ and $E_3,E_4$, respectively. Then we can show that
there is a coordinate neighborhood $(x_1,x_2,x_3,x_4)$  of each point $p$ such that $D^1$ is tangent to the level sets $\{(x_1,x_2,x_3,x_4)|x_3,x_4\textrm{ constants}\}$ and
$D^2$ is tangent to the level sets $\{(x_1,x_2,x_3,x_4)|x_1,x_2\textrm{ constants}\}$ \cite{JS}. Thus we can get the metric description for $g$ as follows;
\begin{equation}
  g=g_{11}dx_1^2+g_{12}dx_1\odot dx_2+g_{22}dx_2^2+g_{33}dx_3^2+g_{34}dx_3\odot dx_4+g_{44}dx_4^2 \label{gg}
\end{equation}
Now through a couple steps, we will show that $g$ can be written on a neighborhood of each point as
\begin{equation}
  g=ds^2+p^2(s)dt^2+h^2(s)\tilde{g}.  \label{goalofg}
\end{equation}
We refine the argument in \cite{JS}.
\begin{lemma}\label{metric1}
 Suppose $[E_1,E_2]=-\eta(s) E_2$ for some smooth function $\eta(s)$ and $[E_i,E_j]\in D^2$ for $i=3,4$ and $j=1,\cdots,4$. Then the metric $g$ of (\ref{gg}) can be written as
  \begin{equation}
   g=ds^2+p(s)^2dt^2+g_{33}dx_3^2+g_{34}dx_3\odot dx_4+g_{44}dx_4^2 \label{gg2}
  \end{equation}
  where $p(s)$ is a smooth function, $g_{ij}$ are functions of $(x_1,x_2,x_3,x_4)$ and $E_1=\frac{\partial}{\partial s}$, $E_2=\frac{1}{p}\frac{\partial}{\partial t}$.
\end{lemma}
\begin{proof}
  As $E_1=\nabla s$, we have $ds=g(E_1,\cdot)$. Define a 1-form $\omega_2(\cdot):=g(E_2,\cdot)$. Then one can easily check that $ds^2+\omega_2^2=
  g_{11}dx_1^2+g_{12}dx_1\odot dx_2+g_{22} dx_2^2$. Now define a function $p(s):=e^{\int_{s_0}^s \eta(u) du}$ for a constant $s_0$, so that
  $\eta=\frac{p'}{p}$. Then $d(\frac{\omega_2}{p})(E_1,E_2)=0$ and for $i\in\{3,4\}$, $j\in\{1,\cdots,4\}$, $d(\frac{\omega_2}{p})(E_i,E_j)=-\frac{dp\wedge \omega_2}{p^2}(E_i,E_j)+
  \frac{1}{p}d\omega_2(E_i,E_j)=\frac{1}{p}d\omega_2(E_i,E_j)=-\frac{1}{p}\omega_2([E_i,E_j])=0$. Thus, $d(\frac{\omega_2}{p})=0$ and $\frac{\omega_2}{p}=dt$ for some function $t$
  modulo a constant in a neighborhood of $p$. Therefore we can write $g$ as (\ref{gg2}).
\end{proof}

\begin{lemma}\label{metric2}
  Suppose $D^2$ is totally umbilic, $<\nabla_{E_3}E_3,E_1>=-\zeta(s)$ for some smooth function $\zeta(s)$ and $<\nabla_{E_i}E_j,E_2>=0$ for $i,j\in\{3,4\}$. Then the metric $g$ of (\ref{gg2}) can be written as
  \begin{equation}
    ds^2+p^2(s)dt^2+h^2(s)\tilde{g}
  \end{equation}
  where $h(s)$ is a smooth function and $\tilde{g}$ is (a pull-back of) a Riemannian metric on a 2-dimensional domain with $x_3,x_4$ coordinates.
\end{lemma}
\begin{proof}
  We use coordinates $(s,t,x_3,x_4)$ so that $\partial_1=\frac{\partial}{\partial s}$ and $\partial_2=\frac{\partial}{\partial t}$. Since $\partial_3$ and $\partial_4$ are both
  of the form $\gamma E_3+\delta E_4$, we have that $<\nabla_{\partial_i}\partial_j,\partial_2>=0$ for $i,j\in\{3,4\}$ by our assumption. Then for $i,j\in\{3,4\}$, $0=<\nabla_{\partial_i}\partial_j,\partial_2>=<\Gamma_{ij}^k\partial_k, \partial_2>=-\frac{1}{2}\partial_2 g_{ij}=-\frac{1}{2}\frac{\partial g_{ij}}{\partial t}$.

  Now we consider the second fundamental form of a leaf for $D^2$ with respect to $E_1$; $H^{E_1}(u,u)=-<\nabla_u u,E_1>$. As $D^2$ is totally umbilic,
  $H^{E_1}(u,u)=\phi g(u,u)$ for some function $\phi$ and any $u$ tangent to $D^2$. Then, $H^{E_1}(E_3,E_3)=-<\nabla_{E_3}E_3,E_1>=\zeta$ so $\phi=\zeta$ which is a function
  of $s$ only. For $i,j\in\{3,4\}$, $ \zeta g_{ij}=H^{E_1}g(\partial_i,\partial_j)=-<\nabla_{\partial_i}\partial_j,\frac{\partial}{\partial s}>=-<\Gamma_{ij}^k\partial_k,\frac{\partial}{\partial s}>=\frac{1}{2}\frac{\partial}{\partial s}g_{ij}$. Thus for $i,j\in\{3,4\}$, we get $g_{ij}=e^{C_{ij}}h(s)^2$.
  Here the function $h(s)>0$ depends only on $s$ and for each function $C_{ij}$ depends only on $x_3,x_4$.
\end{proof}

Due to Lemma \ref{77}, our adapted frame field $\{E_i\}$ satisfy the assumptions of Lemma \ref{metric1} and \ref{metric2}. So we proved that our $g$ can be written as $(\ref{goalofg})$.
 
 \bigskip
 It is already known that if the metric $g=ds^2+p(s)^2dt^2+h(s)^2\tilde{g}$ satisfies the harmonic Weyl curvature condition, then the two dimensional
metric $\tilde{g}$ has constant curvature, say $k$ (see e.g \cite{Derd}). Now we can set $E_1 = \frac{\partial }{\partial  s}$, $E_2 = \frac{1}{p(s)}\frac{\partial }{ \partial  t}$,  $E_3 =  \frac{1}{h(s)} e_3$ and $E_4 =  \frac{1}{h(s)} e_4$, where $e_3$ and $e_4$ are orthonormal frame fields of the Riemannian metric $ \tilde{g}$ on a domain in the ($x_3, x_4$)-plane. And we can compute the coefficients $\zeta_i, q, r   $ in the formula of Lemma \ref{77} by computing covariant derivative of the metric $g=ds^2+p(s)^2dt^2+h(s)^2\tilde{g}$. Then we get $\zeta_2=  \frac{p^{'}}{p}$, $\zeta_3=\zeta_4=  \frac{h^{'}}{h}$ and $q=r=0$. Now we can compute curvature components from
the coefficients.

\begin{align*}
&  R_{1ii1}=-\zeta_i^{'}  -  \zeta_i^2 ,\textrm{ for }i=2,3,4\\
&R_{2332} = R_{2442} = -\zeta_2 \zeta_3=- \frac{p^{'}}{p} \frac{h^{'}}{h} \\
&R_{3443} =  -\zeta_3\zeta_4+\frac{k}{h^2}=-\zeta_3^2-X,\textrm{ where }X=-\frac{k}{h^2}
\end{align*}
\begin{align}
R_{11}  =& - \zeta_2^{'}-\zeta_2^2 -2\zeta_3^{'}-2\zeta_3^2 \nonumber  \\
R_{22} = & -\zeta_2^{'}-\zeta_2^2 -2\zeta_2 \zeta_3 \nonumber\\
R_{33} =&R_{44} =-\zeta_3^{'}-2\zeta_3^2 -\zeta_2 \zeta_3  -X     \label{riccicurvature}\\
R_{ij} =&0, \quad {\rm  if} \ \ \  i \neq j \nonumber\\
R \ \ =&-2\zeta_2'-4\zeta_3'-2\zeta_2^2-6\zeta_3^2-4\zeta_2\zeta_3-2X \nonumber
\end{align}

Now our goal is finding out what is $\zeta_i$. To do that, we put these expressions of curvatures into harmonic Weyl equation and basic equations derived in section 2.
Then we will get descriptions of $\zeta_i'$ and $X$ containing only $\zeta_2$ and $\zeta_3$.

\begin{lemma}\label{primeexpall}
 If $Q:=\zeta_3(m-1)(4\rho-1)+\zeta_2\{4\rho-1+m(2\rho-1)\}\neq 0$, then  we have the followings

  \begin{align}
 & X=\frac{1}{Q}(\zeta_3-\zeta_2)\Big[\zeta_3^2(m-1)(1-4\rho)-\lambda(m+1)+2\zeta_2\zeta_3\big\{4\rho-1+m(5\rho-1)\big\}\Big]
\label{expofx112}\\
&\zeta_2'=\frac{1}{Q}\Big[\zeta_3\lambda(m-1)+\zeta_2^2\zeta_3\big\{1-2(m+2)\rho\big\}+  \zeta_2^3\big\{1-4\rho  \nonumber\\
&\quad\qquad\qquad\qquad\qquad+m(1-2\rho)\big\}+\zeta_2\big\{\lambda-2\zeta_3^2(m-1)(4\rho-1)\big\}\Big]\label{zeta2prime}\\
&\zeta_3'=\frac{1}{Q}\zeta_3\Big[\lambda m-\zeta_3^2(m-1)(4\rho-1)-\zeta_2^2(1+m-4\rho-2m\rho) \nonumber\\
&\qquad\qquad\qquad\qquad\qquad-\zeta_2\zeta_3\{8\rho-2+m(10\rho-3)\})\Big]
\label{zeta3prime}
\end{align}

\end{lemma}

\begin{proof}
 To prove this lemma we derive three equations which contain only $\zeta_i,\zeta_i'$ and $X$. First,
from the equation (\ref{R1ii1}),  we get two equations for $i=2,3$.
\begin{eqnarray}
  (m+1)\zeta_2'+(m+1)\zeta_2^2+2\zeta_2\zeta_3+\rho R+\lambda+\frac{R'}{f'}m(\frac{1}{6}-\rho)=0   \label{r1ii1-2}\\
  (m+1)\zeta_3'+(m+2)\zeta_3^2+\zeta_2\zeta_3+X+\rho R+\lambda+\frac{R'}{f'}m(\frac{1}{6}-\rho)=0   \label{r1ii1-3}
\end{eqnarray}
Now subtract (\ref{r1ii1-3}) from (\ref{r1ii1-2}), then we get an equation which does not contain $R'$-term.
\begin{equation}
  (m+1)\zeta_2'+(m+1)\zeta_2^2-(m+1)\zeta_3'-(m+2)\zeta_3^2+\zeta_2\zeta_3-X=0 \label{r1ii1}
\end{equation}
This is the first equation. The second one follows from $\zeta_i=\frac{1}{f'}(\rho R+\lambda-R_{ii})$. We have three equations for $i=1,2,3$.
\begin{eqnarray}
  -\zeta_2'-\zeta_2^2-2\zeta_2\zeta_3+f'\zeta_2=\rho R+\lambda  \label{quasi-Einstein manifold2}\\
  -\zeta_3'-2\zeta_3^2-\zeta_2 \zeta_3-X+f'\zeta_3=\rho R+\lambda   \label{quasi-Einstein manifold3}
\end{eqnarray}

Multiply $\zeta_2$ and $\zeta_3$ to (\ref{quasi-Einstein manifold3}) and (\ref{quasi-Einstein manifold2}) respectively, and subtract each other. Then we obtain the second equation.
\begin{equation}
  (\zeta_2-\zeta_3)\{\rho(2\zeta_2^2+4\zeta_2\zeta_3+6\zeta_3^2+2X+2 \zeta_2'+4 \zeta_3')-\lambda\}=\zeta_2X+\zeta_2\zeta_3'-\zeta_2'\zeta_3 \label{riccicurv}
\end{equation}
To get the last equation, we need to use harmonic Weyl curvature condition $\nabla_1R_{ii}-\nabla_iR_{1i}=\frac{R'}{6}$.  Putting (\ref{riccicurvature}) into this equation, we get
\begin{eqnarray}
  -\zeta_2''=2\zeta_2'\zeta_2+2\zeta_2'\zeta_3+2\zeta_2^2\zeta_3-2\zeta_2\zeta_3^2+\frac{R'}{6}    \label{double2}\\
  -\zeta_3''=3\zeta_3'\zeta_3+\zeta_2\zeta_3'+\zeta_2\zeta_3^2-\zeta_2^2\zeta_3-\zeta_3X+\frac{R'}{6}    \label{double3}
\end{eqnarray}

Now differentiate (\ref{r1ii1}), and using (\ref{double2}) and (\ref{double3}), eliminate double prime terms. Then we can obtain the last one.
\begin{equation}
  (2m+1)\zeta_2'\zeta_3-(m+2)\zeta_2\zeta_3'-(m-1)\zeta_3(\zeta_3'-X)+3(m+1)\zeta_2\zeta_3(\zeta_2-\zeta_3)=0 \label{lastone}
\end{equation}

Now if we assume $\zeta_3(m+1)(4\rho-1)+\zeta_3\{1-4\rho+m(2\rho-1)\}\neq 0$, then we are able to express $\zeta_2',\zeta_3'$ and $X$ in terms of $\zeta_2,\zeta_3$ by equating (\ref{r1ii1}), (\ref{riccicurv}) and (\ref{lastone}).

 $X=\frac{(\zeta_3-\zeta_2)\Big[\zeta_3^2(m-1)(1-4\rho)-\lambda(m+1)+2\zeta_2\zeta_3\big\{4\rho-1+m(5\rho-1)\big\}\Big]}{\zeta_3(m-1)(4\rho-1)+\zeta_2\{4\rho-1+m(2\rho-1)\}} $

  $\zeta_2'=\frac{\zeta_3\lambda(m-1)+\zeta_2^2\zeta_3\big\{1-2(m+2)\rho\big\}+\zeta_2^3\big\{1-4\rho+m(1-2\rho)\big\}+\zeta_2\big\{\lambda-2\zeta_3^2(m-1)(4\rho-1)\big\}}{\zeta_3
(m-1)(4\rho-1)+\zeta_2\{4\rho-1+m(2\rho-1)\}}$

  $\zeta_3'=\frac{\zeta_3\big[\lambda m-\zeta_3^2(m-1)(4\rho-1)-\zeta_2^2(1+m-4\rho-2m\rho)-\zeta_2\zeta_3\{8\rho-2+m(10\rho-3)\}\big]}{\zeta_3(m-1)(4\rho-1)+\zeta_2\{4\rho-1+m(2\rho-1)\}}$
\end{proof}

 Now as mentioned in the introduction, we will get a number of possible relations between $\zeta_2$ and $\zeta_3$ by using Lemma \ref{primeexpall}.
 In the rest of this section, we assume $m\neq\pm1$. When $m=\pm1$, many coefficient in our argument become the zero or the infinity. Thus our method does not work in that case.
\begin{lemma}\label{five}
  Suppose $\rho\notin\{\frac{1}{4},\frac{1}{6}\}$ and $m\notin \{\pm1,-2\}$. Then $\zeta_2$ and $\zeta_3$ satisfy at least one of the followings.
\begin{footnotesize}
\begin{align}
  & {\rm 1)}\quad \zeta_3=0 \label{rhocase2}\\
 & {\rm 2)}\quad 3(4\rho-1)\zeta_2\zeta_3-\lambda=0 \label{rhocase3}\\
 &  {\rm 3)}\quad \lambda(m-1)(3\rho-1)+\zeta_2\zeta_3(m-1)(4\rho-1) +\zeta_2^2(9\rho-2)\{4\rho-1+m(2\rho-1)\}=0 \label{rhocase4}\\
&{\rm 4)}\quad (4\rho-1)(m-1)\zeta_3+\{4\rho-1+m(2\rho-1)\}\zeta_2=0 \label{rhocase5}
\end{align}
\end{footnotesize}
\end{lemma}

\begin{proof}
Suppose $\zeta_i$ satisfy $\zeta_3(m-1)(4\rho-1)+\zeta_2\{4\rho-1+m(2\rho-1)\}\neq 0$. Then we can use the expressions in the above Lemma.
We differentiate (\ref{expofx112}) with the local variable $s$ and get rid of $\zeta_i'$-terms by using (\ref{zeta2prime}) and (\ref{zeta3prime}).
Then we obtain an expression of $X'$ which contains only $\zeta_2$ and $\zeta_3$.
But note that from the definition, $X'=\Big(-\frac{k}{h^2}\Big)'=-2\zeta_3X$. Comparing these two equations of $X'$, we can get
\begin{footnotesize}
\begin{align*}
 & \frac{2(\zeta_2-\zeta_3)\zeta_3m(m+1)\{3\zeta_2\zeta_3(4\rho-1)-\lambda\}}{\{
\zeta_3(m-1)(4\rho-1)+\zeta_2(4\rho-1+m(2\rho-1))\}^3}\Big\{\lambda(m-1)(3\rho-1)\\
&\qquad\qquad\qquad+\zeta_2\zeta_3(m-1)(4\rho-1)+\zeta_2^2(9\rho-2)(4\rho-1+m(2\rho-1))\Big\}=0
\end{align*}
\end{footnotesize}
As mentioned in section 2, $g$ and $f$ are real analytic in harmonic coordinate. Thus $f'=|\nabla f|$ and the Ricci eigenvalues $\lambda_i$ are real analytic in $M_{\mathcal{T}}\cap\{\nabla f\neq 0\}$. Then $\zeta_i=\frac{1}{f'}(\rho R+\lambda-R_{ii})$ is also real analytic. Thus to satisfy the above equation, one of factors of the left side
should be zero. But by our assumption $\zeta_2$ is not equal to $\zeta_3$ and $m$ cannot be $-1$ or $0$.
Therefore we can get 1),2) and 3).
\end{proof}

 We are going to derive a description of the Riemannian metric $g$ for each four cases. For convenience, we first consider $\zeta_2=0$. And then
 we will assume $\zeta_2\neq0$ in other cases.

\begin{lemma}\label{case1}
 If $\zeta_2=0$, $\zeta_3\neq 0$, then $R_{22}=R_{33}$ which is a contradiction.
\end{lemma}
\begin{proof}
If $\zeta_3$ is also zero, then the metric is an Einstein. Suppose $\zeta_3\neq 0$. Then from (\ref{lastone}), we get $-(m-1)\zeta_3(\zeta_3'-X)=0$. Since $m\neq1$, we obtain $\zeta_3'=X$.
Putting this in (\ref{r1ii1}), then we get $\zeta_3'+\zeta_3^2=0$. Then we have $R_{1221}=R_{1331}=0$ which means that $R_{22}=R_{33}$. Thus this is a contradiction.
\end{proof}

\begin{lemma}\label{112core}
   If $\zeta_3=0$ and $\zeta_2\neq 0$
then for each point $p$ in an open set $O\subset M_A\cap\{\nabla f\neq 0\}$, there exists a neighborhood $V$ of $p$ in $O$ which can be one of the following; Here $C$ is an arbitrary constant.

   {\rm (i)} $(V,g)$ is isometric to a domain in $\mathcal{B}_{\frac{R}{2(m+2)}}^2 \times\mathbb{N}_{\frac{R(m+1)}{2(m+2)}}^2$ with $g=ds^2+\sin^2(\sqrt{\frac{R}{2(m+2)}}s)dt^2+\tilde{g}$, $f=-m\ln(\cos\sqrt{\frac{R}{2(m+2)}}s)+C$ and $\frac{R}{2(m+2)}>0$. Here $s\in(0,\frac{\pi}{2})$ is the distance function on $\mathcal{B}_{\frac{R}{2(m+2)}}^2$ from the north pole and $\tilde{g}$ has constant curvature $\frac{R(m+1)}{2(m+2)}$.

 {\rm (ii)} $(V,g)$ is isometric to a domain in $\mathcal{D}_{\frac{R}{2(m+2)}}^2 \times\mathbb{N}_{\frac{R(m+1)}{2(m+2)}}^2$ with $g=ds^2+\cosh^2(\sqrt{-\frac{R}{2(m+2)}}s)dt^2+\tilde{g}$, $f=-m\ln(-\sinh\sqrt{-\frac{R}{2(m+2)}}s)+C$ and $\frac{R}{2(m+2)}<0$. Here $s\in(-\infty,0)$ is the signed distance function on $\mathcal{D}_{\frac{R}{2(m+2)}}^2$ from the line $\{(s,t)|s=0\}$ and $\tilde{g}$ has constant curvature $\frac{R(m+1)}{2(m+2)}$.

 {\rm (iii)} $(V,g)$ is isometric to a domain in $\mathbb{H}_{\frac{R}{2(m+2)}}^2 \times\mathbb{N}_{\frac{R(m+1)}{2(m+2)}}^2$ with $g=ds^2+\sinh^2(\sqrt{-\frac{R}{2(m+2)}}s)dt^2+\tilde{g}$, $f=-m\ln(\cosh\sqrt{-\frac{R}{2(m+2)}}s)+C$ and $\frac{R}{2(m+2)}<0$. Here $s\in(-\infty,0)$ is the signed distance function on $\mathbb{H}_{\frac{R}{2(m+2)}}^2$ from the point $\{s=0\}$ and $\tilde{g}$ has a constant curvature $\frac{R(m+1)}{2(m+2)}$.

where $R$ is a constant scalar curvature. In particular, we have $R=\frac{-2\lambda(m+2)}{4\rho-1+m(2\rho-1)}$ when $4\rho-1+m(2\rho-1)\neq 0$.
\end{lemma}
\begin{proof}
  Suppose $\zeta_2\neq0$. From (\ref{r1ii1}), we have $(m+1)\zeta_2'+(m+1)\zeta_2^2-X=0$. And from (\ref{riccicurv}), we get $2\rho(\zeta_2'+\zeta_2^2)+(2\rho-1)X-\lambda=0$. Equating these two equation, we get
\begin{equation}
  \{4\rho-1+m(2\rho-1)\}(\zeta_2'+\zeta_2^2)-\lambda=0 \label{223}
\end{equation}
If $4\rho-1+m(2\rho-1)\neq0$ and $\lambda=0$, then $g$ is an Einstein metric. If $4\rho-1+m(2\rho-1)\neq0$ and $\lambda\neq0$, then we get
\begin{align*}
  \zeta_2=\left\{
      \begin{array}{ll}
        -\sqrt{\Lambda}\tan(\sqrt{\Lambda}s+C_1), & \hbox{$\Lambda>0$;} \\
        \sqrt{-\Lambda}\tanh(\sqrt{-\Lambda}s+C_2), & \hbox{$\Lambda<0$;} \\
        \sqrt{-\Lambda}\coth(\sqrt{-\Lambda}s+C_3), & \hbox{$\Lambda<0$;}
           \end{array}
    \right.
\end{align*}
where $\Lambda=\frac{-\lambda}{4\rho-1+m(2\rho-1)}$ and $C_i$ are arbitrary constants. For convenience, choose $C_1=-\frac{\pi}{2}$ and $C_2=C_3=0$.  Then we have
\begin{align*}
  \zeta_2=\left\{
      \begin{array}{ll}
        \sqrt{\Lambda}\cot(\sqrt{\Lambda}s), & \hbox{$\Lambda>0$;} \\
        \sqrt{-\Lambda}\tanh(\sqrt{-\Lambda}s), & \hbox{$\Lambda<0$;} \\
        \sqrt{-\Lambda}\coth(\sqrt{-\Lambda}s), & \hbox{$\Lambda<0$;}
           \end{array}
    \right.
\end{align*}
Since we have $f'\zeta_2=m\Lambda$ from (\ref{quasi-Einstein manifold2}), we can obtain $f$ and $g$ as follows
\begin{align*}
  (g,f)=\left\{
          \begin{array}{ll}
           g=ds^2+\sin^2({\sqrt{\Lambda}s})dt^2+\tilde{g},\quad f=-m\ln(\cos\sqrt{\Lambda}s)+C & \hbox{$\Lambda>0$;} \\
            g=ds^2+\cosh^2({\sqrt{-\Lambda}s})dt^2+\tilde{g},\quad f=-m\ln(-\sinh\sqrt{-\Lambda}s)+C, & \hbox{$\Lambda<0$;} \\
            g=ds^2+\sinh^2({\sqrt{-\Lambda}s})dt^2+\tilde{g},\quad f=-m\ln(\cosh\sqrt{-\Lambda}s)+C, & \hbox{$\Lambda<0$.}
          \end{array}
        \right.
\end{align*}
Note that $R_{1221}=\Lambda,R_{3443}=(m+1)\Lambda$ and $R=2(m+2)\Lambda$. The curvatures of the space are constants which depend
on the sign of $\Lambda$. Since we assumed that $f'=|\nabla f|> 0$, if $f=-m\ln(\cos\sqrt{\Lambda}s)+C$,
then $f$ can be defined only on $(0,\frac{\pi}{2})$. Thus from the fact $R_{1221}=\Lambda$, we can say that two dimensional part $ds^2+\sin^2(\sqrt{\Lambda}s)dt^2$ of the metric $g$ is a 2-dimensional disk which is contained in $\mathbb{S}_{\Lambda}$. By the
same argument, we can obtain (ii) and (iii).

 Now suppose $4\rho-1+m(2\rho-1)=0$. Then $\lambda=0$ and $X=-\rho R$. Since $\zeta_3=\frac{h'}{h}$ is a zero, $X=-\frac{k}{h^2}$ is a constant.
Thus $R$ is a constant function. Then we obtain $\zeta_2'+\zeta_2^2+\frac{\rho R}{m+1}=0$ from (\ref{r1ii1-2}). Since we assumed $\rho=\frac{m+1}{2m+4}$, this equation is equal to
$\zeta_2'+\zeta_2^2+\frac{R}{2(m+2)}=0$. But note that we can regard (\ref{223}) as $\zeta_2'+\zeta_2^2+\frac{R}{2(m+2)}=0$. Thus we can obtain general $g$ and $f$ including all cases.
\end{proof}

Since all three cases in Lemma \ref{112core} have constant scalar curvature, they also satisfy harmonic curvature condition $d^{\nabla} Rc=0$.

\begin{lemma}\label{3331}
  If $3(4\rho-1)\zeta_2\zeta_3-\lambda=0$, then $R_{22}=R_{33}$ which is a contradiction.
\end{lemma}
\begin{proof}
  By assumption, $\rho$ cannot be $\frac{1}{4}$. Putting the relation $\zeta_2\zeta_3=\frac{\lambda}{3(4\rho-1)}$ in (\ref{zeta2prime}) and (\ref{zeta3prime}) respectively, we get
\begin{align*}
  \zeta_2'+\zeta_2^2=\zeta_3'+\zeta_3^2=\frac{\lambda}{3(4\rho-1)}
 \end{align*}
But this means that $R_{22}=R_{33}$ which is a contradiction.
\end{proof}

Note that if $\zeta_2$ and $\zeta_3$ are both nonzero constants, then $R$ and $X$ are also constants. Then we have $2\zeta_2\zeta_3(\zeta_2-\zeta_3)=0$ from (\ref{double2}).
This is a contradiction to the fact that $\zeta_2\neq \zeta_3$. In the next Lemma, we will show that to satisfy (\ref{rhocase4}), $\rho$ must be zero by using this fact.

\begin{lemma}
 If $\zeta_2$ and $\zeta_3$ satisfy (\ref{rhocase4}), then  for each point $p$ in an open set $O\subset M_\mathcal{T}\cap\{\nabla f\neq 0\}$, there exists a neighborhood $V$ of $p$ in $O$ which is isometric to a domain in $\mathbb{R}^4=\{s,t,x_3,x_4\}\setminus \{s\leq0\}$ with the Riemannian metric $g=ds^2+s^{\frac{2(m-1)}{3(m+1)}}dt^2+s^\frac{4}{3}\hat{g}$ and $f=\frac{2m}{3(m+1)}\ln s+C$ with $\rho=\lambda=0$ and $R=-\frac{4m(m-1)}{9(m+1)^2s^2}$. Here $m>0$ or $m<-1$.
\end{lemma}
\begin{proof}
 Suppose nonzero $\zeta_2$ and $\zeta_3$ satisfy (\ref{rhocase4}).  Taking a derivative of (\ref{rhocase4}), we get
\begin{equation}
  (\zeta_2'\zeta_3+\zeta_2\zeta_3')(m-1)(4\rho-1)+2\zeta_2 \zeta_2'(9\rho-2)\{4\rho-1+m(2\rho-1)\}=0 \nonumber
\end{equation}

 First using expressions of $\zeta_2'$ and $\zeta_3'$, we get an equation which contains only $\zeta_2$ and $\zeta_3$. Then put the relation  $\zeta_3=\frac{\zeta_2^2(2-9\rho)\{4\rho-1+m(2\rho-1)\}-\lambda(m-1)(3\rho-1)}{\zeta_2(m-1)(4\rho-1)}$ in that equation. Then we can get

  $\frac{(9\rho-2)\big\{\lambda(m-1)+3\zeta_2^2(4\rho-1+m(2\rho-1))\big\}\big\{\lambda(m-1)(1-3\rho)+\zeta_2^2\rho(12\rho-3+m(6\rho-1))\big\}}{\zeta_2(m-1)(4\rho-1)}=0$

When $\rho=\frac{2}{9}$, we can get $\zeta_2\zeta_3+3\lambda=0$. Then we obtain a contradiction $R_{22}=R_{33}$ by same argument as in Lemma \ref{3331}. Now suppose $\lambda(m-1)+3\zeta_2^2\{4\rho-1+m(2\rho-1)\}=0$. If $\rho=\frac{m+1}{2(m+2)}$, then
$\lambda$ should be zero. But this means that $\zeta_2\zeta_3=0$. Thus $\rho$ cannot be $\frac{m+1}{2(m+2)}$. Hence we can say $\zeta_2=\pm\sqrt{\frac{\lambda(1-m)}{3(4\rho-1+m(2\rho-1))}}$ with $\frac{\lambda(1-m)}{3(4\rho-1+m(2\rho-1))}>0$. Then $\zeta_3$ is also a constant
$\zeta_3=\pm\frac{\lambda\sqrt{4\rho-1+m(2\rho-1)}}{(4\rho-1)\sqrt{3\lambda(1-m)}}$. Therefore this case cannot happen.

Consider the case $\lambda(m-1)(1-3\rho)+\zeta_2^2\rho\{12\rho-3+m(6\rho-1)\}=0$. We can easily see that if $\rho=\frac{m+3}{6(m+2)}$, then we get $\zeta_2=0$ or $\zeta_2=\zeta_3$. Suppose $\rho=0$. Then $\lambda$ also zero, so we get $\zeta_2=\frac{m-1}{2(m+1)}\zeta_3$ from (\ref{rhocase4}). We can obtain $X=0$ from (\ref{riccicurv}) and then (\ref{r1ii1}) is reduced to $b'+\frac{3}{2}b^2=0$. Therefore we can get $\zeta_2=\frac{m-1}{3(m+1)s},\zeta_3=\frac{2}{3s}$ and $R=-\frac{4m(m-1)}{9(m+1)^2s^2}$ where $s$ is defined modulo by constant. Thus we obtain  $g=ds^2+s^{\frac{2(m-1)}{3(m+1)}}dt^2+s^\frac{4}{3}\hat{g}$ and $f=\frac{2m}{3(m+1)}\ln s+C$. And to satisfy $f'=|\nabla f|>0$, $m>0$ or $m<-1$.

 Now assume $\rho$ is not zero. Then we can say that
$\zeta_2=\pm\sqrt{\frac{\lambda(m-1)(3\rho-1)}{\rho(12\rho-3+m(6\rho-1))}}$. Then $\zeta_3$ is also a nonzero constant. Thus this case is also impossible.
\end{proof}

\begin{lemma}\label{case5}
  Nonzero $\zeta_2$ and $\zeta_3$ cannot satisfy (\ref{rhocase5}).
\end{lemma}
\begin{proof}
  Suppose $\zeta_2$ and $\zeta_3$ satisfy $(4\rho-1)(m-1)\zeta_3+(4\rho-1+m(2\rho-1))\zeta_2=0$.
 Since $\rho\neq\frac{1}{4}$, we have $\zeta_3=\frac{4\rho-1+m(2\rho-1)}{(4\rho-1)(1-m)}\zeta_2$. Putting this in (\ref{r1ii1}), (\ref{riccicurv}) and (\ref{lastone}), we can get
\begin{equation*}
  \zeta_2=\pm\sqrt{\frac{\lambda(m-1)}{3(1-4\rho+m(1-2\rho))}},\quad \zeta_3=\pm\frac{\sqrt{\lambda(1-4\rho+m(1-2\rho))}}{(4\rho-1)\sqrt{3(m-1)}}
\end{equation*}
Therefore this case also give a contradiction.
\end{proof}

We sum up all the results from Lemma \ref{case1} to Lemma \ref{case5}, then we can obtain the following Proposition.
\begin{prop}\label{112result}
  Let $(M^4,g,f)$ be a $(m,\rho)$-quasi-Einstein manifold with harmonic Weyl curvature, $m\notin\{\pm1,-2\}$ and $\rho\notin\{\frac{1}{4},\frac{1}{6}\}$. Suppose that $\lambda_2\neq\lambda_3=\lambda_4$ for an adapted frame fields $E_j$, $j=1,2,3,4$,
 in an open subset $O$ of $M_\mathcal{T}\cap\{\nabla f\neq 0\}$.
Then for each point $p$ in $O$, there exists a neighborhood $V$ of $p$ in $O$  which is isometric to
 one of the following; Here $R$ is the constant scalar curvature of $g$ and $C$ is an arbitrary constant.

{\rm (i)} $g$ is an Einstein metric with $f$ a constant function.

   {\rm (ii)} $(V,g)$ is isometric to a domain in $\mathcal{B}_{\frac{R}{2(m+2)}}^2 \times\mathbb{N}_{\frac{R(m+1)}{2(m+2)}}^2$ with $g=ds^2+\sin^2(\sqrt{\frac{R}{2(m+2)}}s)dt^2+\tilde{g}$, $f=-m\ln(\cos\sqrt{\frac{R}{2(m+2)}}s)+C$ and  $\frac{R}{2(m+2)}>0$. Here $s\in(0,\frac{\pi}{2})$ is the distance function on $\mathcal{B}_{\frac{R}{2(m+2)}}^2$ from the north pole and $\tilde{g}$ has constant curvature $\frac{R(m+1)}{2(m+2)}$.

 {\rm (iii)} $(V,g)$ is isometric to a domain in $\mathcal{D}_{\frac{R}{2(m+2)}}^2 \times\mathbb{N}_{\frac{R(m+1)}{2(m+2)}}^2$ with $g=ds^2+\cosh^2(\sqrt{-\frac{R}{2(m+2)}}s)dt^2+\tilde{g}$, $f=-m\ln(-\sinh\sqrt{-\frac{R}{2(m+2)}}s)+C$ and $\frac{R}{2(m+2)}<0$. Here $s\in(-\infty,0)$ is the signed distance function on $\mathcal{D}_{\frac{R}{2(m+2)}}^2$ from the line $\{(s,t)|s=0\}$ and $\tilde{g}$ has constant curvature $\frac{R(m+1)}{2(m+2)}$.

 {\rm (iv)} $(V,g)$ is isometric to a domain in $\mathbb{H}_{\frac{R}{2(m+2)}}^2 \times\mathbb{N}_{\frac{R(m+1)}{2(m+2)}}^2$ with $g=ds^2+\sinh^2(\sqrt{-\frac{R}{2(m+2)}}s)dt^2+\tilde{g}$, $f=-m\ln(\cosh\sqrt{-\frac{R}{2(m+2)}}s)+C$ and  $\frac{R}{2(m+2)}<0$. Here $s\in(-\infty,0)$ is the signed distance function on $\mathbb{H}_{\frac{R}{2(m+2)}}^2$ from the point $\{s=0\}$ and $\tilde{g}$ has a constant curvature $\frac{R(m+1)}{2(m+2)}$.

{\rm (v)} A domain in $\mathbb{R}^4=\{s,t,x_3,x_4\}\setminus\{s\leq0\}$ with the metric $g=ds^2+s^{\frac{2(m-1)}{3(m+1)}}dt^2+s^{\frac{4}{3}}\hat{g}$, $f=\frac{2m}{3(m+1)}\ln s+C$ with $\rho=\lambda=0$ where $\hat{g}$ is the Euclidean metric on the $(x_3,x_4)-$plane. Here $m>0$ or $m<-1$.

\bigskip

For the case (ii),(iii) and (iv), if $\rho\neq\frac{m+1}{2(m+2)}$ then $R=\frac{-2\lambda(m+2)}{4\rho-1+m(2\rho-1)}$.
\end{prop}

\section{Proof of Theorems}

\quad In this section we first classify the remaining case of $\lambda_2=\lambda_3=\lambda_4$. We will give this result without a proof. One can find the proof in \cite{JS}. Then summing up this result,  Proposition 1 and 2, we give a local classification of $(m,\rho)$-quasi-Einstein manifolds with harmonic
Weyl curvature.
\begin{prop}\label{prop1,3}
  Let $(M^4,g,f)$ be a $(m,\rho)$-quasi-Einstein manifold with harmonic Weyl curvature, $m\notin\{\pm 1,-2\}$, $\rho\neq \frac{1}{4}$ and non-constant $f$. Suppose that $\lambda_2=
  \lambda_3=\lambda_4\neq\lambda_1$ for an adapted frame field in an open subset $O$ of $M_{\mathcal{T}}\cap\{\nabla f\neq 0\}$.

  Then for each point $p_0$ in $O$, there exists a neighborhood $V$ of $p_0$ in $O$ where $g$ is a warped product;
  \begin{equation}
    g=ds^2+h(s)^2\bar{g} \label{1,3case}
  \end{equation}
  for a positive function $h$, where the Riemannian metric $\bar{g}$ has constant curvature, say $k$. In particular,
  $g$ is locally conformally flat.
\end{prop}
\begin{proof}
  One may follow the argument in the proof of Proposition 7.1 in \cite{JS}.
\end{proof}

For the metric $g$ in (\ref{1,3case}), if we write $\nabla_{E_i}E_1=\zeta E_i$ for $i=2,3,4$, then we get the
following equations
\begin{eqnarray}
  &(m+1)\zeta'+(m+3)\zeta^2+2X+\rho R+\lambda+\frac{m R'}{f'}(\frac{1}{6}-\rho)=0 \label{1,3cond1}\\
  &\zeta'+3\zeta^2+2X+\rho R+\lambda=f'\zeta  \label{1,3cond2}\\
  &-3\zeta'-3\zeta^2+f''-\frac{1}{m}(f')^2=\rho R+\lambda \label{1,3cond3}\\
 &-\zeta''=4\zeta\zeta'-2\zeta X+\frac{R'}{6} \label{1,3cond4}
\end{eqnarray}

\begin{rmk}
  If all $\lambda_i$'s, $i=1,\cdots,4$, are equal, then the metric is Einstein. And if $f$ is not constant, then
the conclusion of Proposition \ref{prop1,3} still holds. In fact, from the section 1 on \cite{Che}, the Einstein metric $g$
becomes locally of the form $g=ds^2+(f'(s))^2\tilde{g}$ where $\tilde{g}$ has constant curvature.
\end{rmk}

\quad Now we shall combine Proposition 1, 2 and 3 to prove a classification of  $(m,\rho)$-Einstein manifold with harmonic Weyl curvature.

\begin{proof1}\label{notwoof}
Combining the results of Proposition 1, 2 and 3, we can get the theorem.
\end{proof1}

From Theorem 1, we can describe complete spaces corresponding to (ii), (iii) and (iv). For the type (v), note that this is actually a $(\lambda,4+m)$-Einstein manifold.
In \cite{KK}, D.Kim and Y.Kim showed that there is a constant $\mu$ such that $\mu=w\Delta w+(m-1)|\nabla w|^2+\lambda w^2$. If we compute
this constant $\mu$ for the type (v) space, then we obtain $\mu=0$. But C.He, P.Petersen and W.Wylie \cite{HPW} proved that
complete $(\lambda,n+m)$-Einstein metrics with $m>1$, $\lambda\geq 0$, and $\mu\leq 0$, are the trivial ones with $\lambda=\mu=0$. Thus
the type (v) in Theorem 1 cannot be a complete metric.

\begin{theorem}\label{completecase}
   A (Cauchy) complete $(m,\rho)$-Einstein manifold $(M^4,g,f)$ with harmonic Weyl curvature, $m\notin\{\pm 1,-2\}$ and $\rho\notin\{\frac{1}{4},\frac{1}{6}\}$  is isometric to one of the following;
Here $R$ is the constant scalar curvature of $g$ and $C$ is an arbitrary constant.

  {\rm (i)} $g$ is an Einstein metric with $f$ a constant function.

  {\rm (ii)} A finite quotient of $\mathcal{B}_{\frac{R}{2(m+2)}}^2 \times\mathbb{N}_{\frac{R(m+1)}{2(m+2)}}^2$ with
  $\frac{R}{2(m+2)}>0$
   and  $f=-m\ln(\cos\sqrt{\frac{R}{2(m+2)}}s)+C$.

{\rm (iii)} A finite quotient of $\mathcal{D}_{\frac{R}{2(m+2)}}^2 \times\mathbb{N}_{\frac{R(m+1)}{2(m+2)}}^2$ with  $\frac{R}{2(m+2)}<0$
 and $f=-m\ln(-\sinh\sqrt{-\frac{R}{2(m+2)}}s)+C$.

{\rm (iv)} A finite quotient of $\mathbb{H}_{\frac{R}{2(m+2)}}^2 \times\mathbb{N}_{\frac{R(m+1)}{2(m+2)}}^2$ with $\frac{R}{2(m+2)}<0$
 and $f=-m\ln(\cosh\sqrt{-\frac{R}{2(m+2)}}s)+C$.

{\rm (v)} $g$ is locally conformally flat.

\end{theorem}

\bigskip

As a corollary, we finish this section by proving the classification of $(m,\rho)$-quasi-Einstein manifolds with harmonic curvature. Since if $(M^4,g,f)$ satisfy the harmonic curvature condition then $R$ is a constant,  we can analyze (\ref{1,3cond1})$\thicksim$(\ref{1,3cond4}).
\begin{theorem}
 Let $(M^4,g,f)$ be a $(m,\rho)$-quasi-Einstein manifold with harmonic curvature, $m\notin\{\pm1,-2\}$ and $\rho\notin\{\frac{1}{4},\frac{1}{6}\}$. Then there exist an open dense subset $U$ of $M$ such that
for each point $p$ in $U$ there exists a neighborhood $V_p$  where $(V_p,g)$ is isometric to one of the following;
Here $R$ is the constant scalar curvature and $C$ is an arbitrary constant.

  {\rm (i)} An Einstein metric ($f$ can be a non-constant function).

   {\rm (ii)} A domain in $\mathcal{B}_{\frac{R}{2(m+2)}}^2 \times\mathbb{N}_{\frac{R(m+1)}{2(m+2)}}^2$ with $g=ds^2+\sin^2(\sqrt{\frac{R}{2(m+2)}}s)dt^2+\tilde{g}$, $f=-m\ln(\cos\sqrt{\frac{R}{2(m+2)}}s)+C$ and  $\frac{R}{2(m+2)}>0$.

 {\rm (iii)} A domain in $\mathcal{D}_{\frac{R}{2(m+2)}}^2 \times\mathbb{N}_{\frac{R(m+1)}{2(m+2)}}^2$ with $g=ds^2+\cosh^2(\sqrt{-\frac{R}{2(m+2)}}s)dt^2+\tilde{g}$, $f=-m\ln(-\sinh\sqrt{-\frac{R}{2(m+2)}}s)+C$ and $\frac{R}{2(m+2)}<0$.

 {\rm (iv)} A domain in $\mathbb{H}_{\frac{R}{2(m+2)}}^2 \times\mathbb{N}_{\frac{R(m+1)}{2(m+2)}}^2$ with $g=ds^2+\sinh^2(\sqrt{-\frac{R}{2(m+2)}}s)dt^2+\tilde{g}$, $f=-m\ln(\cosh\sqrt{-\frac{R}{2(m+2)}}s)+C$ and  $\frac{R}{2(m+2)}<0$.

 {\rm (v)} A domain in $\mathbb{I}\times M_{\frac{\rho R+\lambda}{m+3}}$ where $M_{\frac{\rho R+\lambda}{m+3}}$ is a 3-dimensional manifold with constant curvature $\frac{\rho R+\lambda}{m+3}$, with $g$ and $w$ is one of the followings;
\begin{footnotesize}
\begin{align*}
   \left\{
     \begin{array}{ll}
      \mathbb{I}=(0,\frac{\pi}{2}), g=ds^2+\sin^2\Big(\sqrt{\frac{\rho R+\lambda}{m+3}}s\Big)\bar{g},
  f=-m\ln(\cos\sqrt{\frac{\rho R+\lambda}{m+3}}s)+C, & \hbox{$\frac{\rho R+\lambda}{m+3}>0$;} \\
      \mathbb{I}=(-\infty,0),g=ds^2+\cosh^2\Big(\sqrt{\frac{-(\rho R+\lambda)}{m+3}}s\Big)\bar{g},
  f=-m\ln(-\sinh\sqrt{\frac{-(\rho R+\lambda)}{m+3}}s)+C, & \hbox{$\frac{\rho R+\lambda}{m+3}<0$;} \\
      \mathbb{I}=(-\infty,0), g=ds^2+\sinh^2\Big(\sqrt{\frac{-(\rho R+\lambda)}{m+3}}s\Big)\bar{g},
 f=-m\ln(\cosh\sqrt{\frac{-(\rho R+\lambda)}{m+3}}s)+C, & \hbox{$\frac{\rho R+\lambda}{m+3}<0$;}
     \end{array}
   \right.
\end{align*}
\end{footnotesize}
where $\bar{g}$ has constant sectional curvature $k_{\bar{g}}=\frac{\rho R+\lambda}{m+3}$.

\end{theorem}
\begin{proof3}
  The metric $g=ds^2+s^{\frac{2(m-1)}{3(m+1)}}dt^2+s^\frac{4}{3}\hat{g}$ in Theorem 1 does not have a constant scalar curvature, so
  it cannot satisfy the harmonic curvature condition. And we already mentioned that type (ii)$\sim$(iv) satisfy the harmonic curvature condition.

Now suppose the metric $g$ of the type (vi) satisfy the harmonic curvature condition and $\lambda_1\neq\lambda_2=\lambda_3=\lambda_4$. Then (\ref{1,3cond1}) becomes
\begin{equation}
  (m+1)\zeta'+(m+3)\zeta^2+2X+\rho R+\lambda=0 \label{1,3solv}
\end{equation}
Differentiating this equation and equating with (\ref{1,3cond4}), we can obtain $\zeta'=X$. Putting
this in the above equation (\ref{1,3solv}), we get
\begin{equation}
  \zeta'+\zeta^2+\frac{\rho R+\lambda}{m+3}=0
\end{equation}
By same argument as in Lemma \ref{112core}, we can obtain the results.
\end{proof3}

Note that non-trivial Einstein case is classified in [Proposition3.1, \cite{HPW}]. Even they considered $\rho=0$ case, we can apply their proof with $\bar{k}=\frac{\lambda-(1-4\rho)\alpha}{m}$ when $Rc=\alpha g$.

\section{$(\lambda,4+m)$-Einstein manifolds}

As mentioned in an introduction, we give a classification of $(\lambda,4+m)$-Einstein manifolds. First recall the definition of $(\lambda,n+m)$-Einstein manifold.
\begin{definition}\cite{HPW}
  A Riemannian manifold $(M^n,g)$ is called a $(\lambda,n+m)$-Einstein manifold, if there exists a smooth function $w$ on $M$ which satisfies
\begin{align*}
  \nabla dw\quad&=\quad\frac{w}{m}(Rc-\lambda g)\\
w\quad&>\quad0\quad \textrm{on int$(M)$}\\
w\quad&=\quad0\quad \textrm{on $\partial M$ if $\partial M\neq \phi$}
\end{align*}
where $\lambda$ and $m$ are constants.
\end{definition}

Taking $\rho=0$ and $w=e^{-\frac{f}{m}}$ in Theorem 1, we can obtain a classification of $(\lambda,4+m)$-Einstein manifolds under the harmonic Weyl curvature condition.
Since motivation of this special case comes from the warped product Einstein metric, we only consider $m>1$.

\begin{cor}\label{mainthmcor}
  Let $(M^4,g,w)$ be a $(\lambda,4+m)$-Einstein manifold (not necessarily complete) with harmonic Weyl curvature and $m> 1$. Then there exists an open dense subset $U$ of $M$ such that
for each point $p$ in $U$ there exists a neighborhood $V_p$  where $(V_p,g)$ is isometric to one of the following;
 Here $C$ is an positive arbitrary constant.

{\rm (i)} An Einstein metric with $w$ a constant function.

 {\rm (ii)} A domain in $\mathcal{B}_{\frac{\lambda}{m+1}}^2 \times\mathbb{S}_{\lambda}^2$ with $w=C\cos\sqrt{\frac{\lambda}{m+1}}s$ and $\lambda>0$
where  $s\in(0,\frac{\pi}{2}]$ is the distance function on $\mathcal{B}^2_\frac{\lambda}{m+1}$ from the the north pole.

 {\rm (iii)}A domain in $\mathcal{D}_{\frac{\lambda}{m+1}}^2 \times\mathbb{H}_{\lambda}^2$ with  $w=-C\sinh\sqrt{\frac{-\lambda}{m+1}}s$ and $\lambda<0$ where $s\in(-\infty,0]$ can be viewed as the signed distance function on $\mathcal{D}^2_\frac{R}{2(m+2)}$ from the line $\{(s,t)|s = 0\}$

 {\rm (iv)} A domain in $\mathbb{H}_{\frac{\lambda}{m+1}}^2 \times\mathbb{H}_{\lambda}^2$ with $w=C\cosh\sqrt{-\frac{\lambda}{m+1}}s$ and $\lambda<0$ where
$s\in(-\infty,0)$ is the signed distance function on $\mathbb{H}^2_\frac{\lambda}{m+1}$ from the point $\{s=0\}$.

{\rm (v)} A domain in $\mathbb{R}^4=\{s,t,x_3,x_4\}\setminus\{s\leq0\}$ with the metric $g=ds^2+s^{\frac{2(m-1)}{3(m+1)}}dt^2+s^{\frac{4}{3}}\hat{g}$, where
$\hat{g}$ is the Euclidean metric on the $(x_3,x_4)-$plane. Here $\lambda=0$ and $w=Cs^{-\frac{2}{3(m+1)}}$.

{\rm (vi)}  A warped product with the metric of the form $g=ds^2+h(s)^2\bar{g}$ for a positive function $h$, where the Riemannian metric $\bar{g}$ has constant curvature. In
particular, $g$ is locally conformally flat.

\end{cor}
Note that since $w$ can be the zero on the boundary, $f$ is allowed to be the infinite. Thus the range of $s$ of (ii) and (iii) is different to $(m,\rho)$-quasi-Einstein. By applying this
local classification, we could classify complete $(\lambda,4+m)$-Einstein manifolds (See Corollary 1).

\bigskip
  In \cite{KK}, D.Kim and Y.Kim showed that there is a constant $\mu$ such that $\mu=w\Delta w+(m-1)|\nabla w|^2+\lambda w^2$.
This constant $\mu$ is the Ricci curvature of the fiber $F$ of the warped product Einstein manifold whose base is $M$. We can
compute $\mu$ for each case in Corollary 1
\begin{align*}
  \mu=\left\{
        \begin{array}{ll}
          \frac{m-1}{m+1}|\lambda|C^2>0, & \hbox{Case (ii),(iii);} \\
          \frac{m-1}{m+1}\lambda C^2<0, & \hbox{Case (iv);}\\
          0, & \hbox{Case (v).}
        \end{array}
      \right.
\end{align*}

   Now we can construct a warped product Einstein metric. If we construct a warped product Einstein metric $g_E$ over $M$ whose boundary is nonempty, then the fiber $F$ must be $\mathbb{S}^m$(\cite{HPW}). Thus if we construct the metric $g_{E}$ over case (ii) or (iii) in Corollary 1, we obtain
\begin{align*}
  g_E=&ds^2+\sin^2\Big(\sqrt{\frac{\lambda}{m+1}}s\Big)dt^2+\tilde{g}+\cos^2\Big(\sqrt{\frac{\lambda}{m+1}}s\Big)g_{\mathbb{S}^m}\\
  g_E=&ds^2+\cosh^2(\Big(\sqrt{\frac{-\lambda}{m+1}}s\Big)dt^2+\tilde{g}+\sinh^2\Big(\sqrt{\frac{-\lambda}{m+1}}s\Big)g_{\mathbb{S}^m}
\end{align*}

We can easily see that this is the product metric of $\mathbb{S}^{m+2}_{\frac{\lambda}{m+1}}\times \mathbb{S}^2_\lambda$ and $\mathbb{H}^{m+2}_{\frac{\lambda}{m+1}}\times \mathbb{H}^2_\lambda$ respectively. If we construct $g_E$ over case (iv) then fiber $F_1$ should have negative Ricci curvatures. And over (v), $F_2$ should be a Ricci-flat manifold.
Thus the metric $g_E$ is
\begin{align*}
  g_E=&ds^2+\sinh^2(\Big(\sqrt{\frac{-\lambda}{m+1}}s\Big)dt^2+\tilde{g}+\cosh^2\Big(\sqrt{\frac{-\lambda}{m+1}}s\Big)g_{F_1}\\
  g_E=&ds^2+s^{\frac{2(m-1)}{3(m+1)}}dt^2+s^{\frac{4}{3}}\hat{g}+s^{-\frac{4}{3(m+1)}}g_{F_2}
\end{align*}
where $Ric_{F_1}<0$ and $Ric_{F_2}=0$.

\bigskip
We finish our paper by stating a classification of  $(\lambda,4+m)$-Einstein manifolds under harmonic curvature.

\begin{cor}
 Let $(M^4,g,w)$ be a $(\lambda,4+m)$-Einstein manifold with harmonic curvature $m>1$. Then there exist an open dense subset $U$ of $M$ such that
for each point $p$ in $U$ there exists a neighborhood $V_p$  where $(V_p,g)$ is isometric to one of the following;
Here $C$ is an arbitrary positive constant.

  {\rm (i)} An Einstein metric ($w$ can be a non-constant function).

  {\rm (ii)} A domain in $\mathcal{B}_{\frac{\lambda}{m+1}}^2 \times\mathbb{S}_{\lambda}^2$ with $w=C\cos\sqrt{\frac{\lambda}{m+1}}s$ and $\lambda>0$.

 {\rm (iii)}A domain in $\mathcal{D}_{\frac{\lambda}{m+1}}^2 \times\mathbb{H}_{\lambda}^2$ with  $w=-C\sinh\sqrt{\frac{-\lambda}{m+1}}s$ and $\lambda<0$.

 {\rm (iv)} A domain in $\mathbb{H}_{\frac{\lambda}{m+1}}^2 \times\mathbb{H}_{\lambda}^2$ with $w=C\cosh\sqrt{\frac{-\lambda}{m+1}}s$ and $\lambda<0$.

 {\rm (v)} A domain in $\mathbb{I}\times M_{\frac{\lambda}{m+3}}$ where $M_{\frac{\lambda}{m+3}}$ is a 3-dimensional manifold with constant curvature $\frac{\lambda}{m+3}$, with $g$ and $w$ is one of the followings;
\begin{align*}
   \left\{
     \begin{array}{ll}
      \mathbb{I}=(0,\frac{\pi}{2}], g=ds^2+\sin^2\Big(\sqrt{\frac{\lambda}{m+3}}s\Big)\bar{g},
  w=C\cos\sqrt{\frac{\lambda}{m+3}}s), & \hbox{$\lambda>0$;} \\
      \mathbb{I}=(-\infty,0],g=ds^2+\cosh^2\Big(\sqrt{\frac{-\lambda}{m+3}}s\Big)\bar{g},
  w=-C\sinh\sqrt{\frac{-\lambda}{m+3}}s), & \hbox{$\lambda<0$;} \\
      \mathbb{I}=(-\infty,0), g=ds^2+\sinh^2\Big(\sqrt{\frac{-\lambda}{m+3}}s\Big)\bar{g},
 w=C\cosh\sqrt{\frac{-\lambda}{m+3}}s), & \hbox{$\lambda<0$;}
     \end{array}
   \right.
\end{align*}

where $\bar{g}$ has constant sectional curvature $k_{\bar{g}}=\frac{\lambda}{m+3}$.

\end{cor}

\end{document}